\documentclass[a4paper,10pt]{scrartcl}

\KOMAoptions{
	twoside=true,
	headsepline=false,	
	headings=normal 	
}

\usepackage{titlesec}	
\titleformat{\section}{\large\bfseries\center}{\thesection}{1em}{}
\titleformat{\subsection}[runin]{\bfseries}{\thesubsection}{1em}{}

\usepackage[activeacute,spanish]{babel}
\usepackage[utf8]{inputenc}
\addto\captionsspanish{} 
\addto\captionsspanish{}

\usepackage{probsoln}

\showanswers

\usepackage{amsmath, amsfonts, amssymb}
\usepackage{multicol}
\usepackage{enumerate}
\usepackage{a4wide}
\usepackage{graphicx}
\usepackage{multicol}
\usepackage[T1]{fontenc}
\usepackage[sc,osf]{mathpazo}
\usepackage[scaled]{beramono}
\usepackage{tgpagella}
\usepackage{hyperref}
\hypersetup{%
    colorlinks=true, linktocpage=true, pdfstartview=FitV,%
    urlcolor=blue, linkcolor=blue, citecolor=blue, pagecolor=blue,%
}

\newtheorem{teo}{Theorem}[section]
\newtheorem{lem}[teo]{Lemma}
\newtheorem{cor}[teo]{Corollary}
\newtheorem{pro}[teo]{Proposition}

\newtheorem{claim}{Claim}
\newtheorem{obs}{Observation}
\newtheorem{ex}{Example}
\newtheorem{defintro}{Definition} 
\numberwithin{equation}{section} 

\usepackage{fancybox,fancyhdr}
\newenvironment{proof}{\trivlist\item[\hskip \labelsep{\emph{\textbf{Proof}}}:]}{\nopagebreak \hfill $\Box$ \endtrivlist}
 
\newenvironment{proofclaim}{\trivlist\item[\hskip \labelsep{\emph{\textbf{Proof of the Claim}}}:]}{\nopagebreak \hfill $\Box$ \endtrivlist}

\usepackage{spalign} 

\usepackage{systeme}

    \newcommand{\Z}{\mathbb{Z}}     
    \newcommand{\R}{\mathbb{R}}     
    \def\S{\mathbb{S}}
    \def\k{\mathfrak{K}}

    \def\K{\mathcal{K}}
    \def\sig{\Sigma}
    \def\s2{\S^2}
    \def\r3{\mathbb{R}^3}
    \def\t{\theta}
    \def\hs5{\hspace{.5cm}}

    {
        \vspace{-.3cm}
        \begin{minipage}[c]{0.05\textwidth}
        \vspace{.3cm}
        \centering
        \raisebox{0.3cm}{\includegraphics[height=0.6cm]{./icon/lupa.png}}
        \end{minipage}
        \begin{minipage}[t]{0.85\textwidth}
        \itshape
    }
    {
        \end{minipage}
    }

\usepackage{pgfplots} 
\usetikzlibrary{positioning}
\usepgfplotslibrary{fillbetween}
\pgfplotsset{compat=1.12}  

\makeatletter
\renewcommand*\env@matrix[1][*\c@MaxMatrixCols c]{%
  \hskip -\arraycolsep
  \let\@ifnextchar\new@ifnextchar
  \array{#1}}
\makeatother  
\usepackage{abstract}

\begin{document}
\thispagestyle{empty}

\begin{center}

\renewcommand{\thefootnote}{\,}
{\fontsize{16}{19}\hspace{.5cm} \textbf{Rotational surfaces of prescribed Gauss curvature in $\mathbb{R}^3$} 
\footnote{
\hspace{.15cm}\emph{Mathematics Subject Classification:} 53A10, 53C42, 34C05, 34C40}}\\
\vspace{0.5cm} { Antonio Bueno$^\dagger$, Irene Ortiz$^\ddagger$}\\
\end{center}
\vspace{.5cm}
$^\dagger$Departamento de Ciencias, Centro Universitario de la Defensa de San Javier, E-30729 Santiago de la Ribera, Spain. \\ \vspace{.3cm}
\emph{E-mail address:} antonio.bueno@cud.upct.es\\
$^\ddagger$Departamento de Ciencias, Centro Universitario de la Defensa de San Javier, E-30729 Santiago de la Ribera, Spain. \\ \vspace{.3cm}
\emph{E-mail address:} irene.ortiz@cud.upct.es

\begin{abstract}
We study rotational surfaces in Euclidean 3-space whose Gauss curvature is given as a prescribed function of its Gauss map. By means of a phase plane analysis and under mild assumptions on the prescribed function, we generalize the classification of rotational surfaces of constant Gauss curvature; exhibit examples that cannot exist in the constant Gauss curvature case; and analyze the asymptotic behavior of strictly convex graphs. We also prove the existence of singular radial solutions intersecting orthogonally the axis of rotation.
\end{abstract}

\section{Introduction}
One of the most outstanding problems in Differential Geometry in the large is the \emph{Minkowski problem for ovaloids} \cite{Min}, whose formulation is stated next:
\begin{quote}
\emph{Let be $\K\in C^1(\S^2),\ \K>0$. Determine the existence of an ovaloid $\sig$ whose Gauss curvature $K_\sig$ is given at each $p\in\sig$ by}
\end{quote}
\begin{equation}\label{defKsup}
K_\sig(p)=\K(\eta_p)
\end{equation}
\begin{quote}
\emph{where $\eta:\sig\rightarrow\S^2$ is the Gauss map.}
\end{quote}
In the words of Eugenio Calabi: \emph{"From the geometric view point it [the Minkowski problem] is the Rosetta Stone, from which several related problems can be solved"}. As Calabi realized, the Minkowski problem has been the cornerstone thanks to which many researches have culminated spectacular advances in many fields. We highlight the works of Minkowski, Alexandrov, Nirenberg, Pogorelov, S.T. Yau and S.Y. Chen \cite{Min,Ale,Nir,Pog1,ChYa}. A necessary and sufficient condition for the existence of such an ovaloid is that $\K>0$ satisfies the following integral condition:
\begin{equation}\label{condicionMinkowski}
\int_{\S^2}\frac{X}{\K(X)}=0.
\end{equation}

In general, the study of surfaces in $\R^3$ defined by a prescribed curvature function in terms of its Gauss map goes back, at least, to the Minkowski and Christoffel problems \cite{Chr}, where in the latter the mean of the curvature radii of the surface are prescribed, i.e. $1/\kappa_1+1/\kappa_2=2\K(\eta)$.

Among the curvature functions that can be prescribed, another having a paramount importance is the mean curvature. The existence and uniqueness of prescribed mean curvature ovaloids has been addressed by many researches, among which we highlight Alexandrov, Pogorelov, Hartman, Wintner, B. Guan, P. Guan, Gálvez and Mira \cite{Ale,Pog2,HaWi,GuGu,GaMi1}. However, the global geometry of complete, non-compact examples was largely unexplored until the first author jointly with Gálvez and Mira \cite{BGM1,BGM2} started to develop \emph{the global theory of surfaces of prescribed mean curvature}, making special emphasis on its relation with constant mean curvature surfaces and the translating solitons of the mean curvature flow. In \cite{BGM1} the authors focused on the structure of properly embedded surfaces, while in \cite{BGM2} the line of inquiry was the description of rotational examples.

In the same spirit as in the development of the theory of prescribed mean curvature surfaces, our goal in this paper is to further analyze the global structure of surfaces satisfying Eq. \eqref{defKsup}, taking as starting point the well-known theory of surfaces of constant Gauss curvature. Specifically, in this paper we focus on the existence and classification of rotational examples, provided that $\K\in C^1(\S^2)$ is \emph{rotationally symmetric}, that is
\begin{equation}\label{dependenciarotacional}
\K(X)=\k(\langle X,e_3\rangle),\hspace{.5cm} \forall X\in\S^2,
\end{equation}
for some $\k\in C^1([-1,1])$. In this case, the following definition arises:
\begin{defintro}\label{def:ksup}
An immersed, oriented surface $\sig$ in $\R^3$ with Gauss map $\eta$ is a $\k$-surface if
\begin{equation}\label{defKsuprot}
K_\sig(p)=\k(\langle\eta_p,e_3\rangle),\hspace{.5cm} \forall p\in\sig.
\end{equation}
\end{defintro}
The quantity $\langle\eta_p,e_3\rangle$ is the so-called \emph{angle function}. Note that Euclidean translations, rotations around vertical lines and reflections respect vertical planes send $\k$-surfaces into $\k$-surfaces. Any such isometry keeps invariant the angle function, hence Eq. \eqref{defKsuprot} holds. In particular, the notion of rotational $\k$-surface is well-defined. Also, if $\sig$ is a $\k$-surface with Gauss map $\eta$ and $\Phi$ is a horizontal reflection, then $\sig^*=\Phi(\sig)$, endowed with $-\eta^*=-d\Phi(\eta)$ as Gauss map, is also a $\k$-surface, since $\langle\eta,e_3\rangle=\langle-\eta^*,e_3\rangle$. Note that changing $\eta^*$ into $-\eta^*$ does not change the Gauss curvature, which is a paramount difference when prescribing the mean curvature.

The most studied case of $\k$-surfaces is, of course, when the function $\k$ is constant. The classification of rotational surfaces of constant Gauss curvature $K\in\R$ is a well-known result, for whose proof three different cases are analyzed: $K=0,\ K<0$ and $K>0$; see e.g. \cite{Car}. Going back to $\k$-surfaces, in such generality for $\k$, it seems hopeless to find an explicit description of the corresponding examples. We overcome this difficulty by treating the ODE fulfilled by the profile curve of a rotational $\k$-surface as a non-linear autonomous system, and derive a qualitative study of its solutions by means of a phase plane analysis. In this way, we will show that under mild assumptions on the prescribed function $\k$, the rotational $\k$-surfaces follow a pattern that resemble us to the surfaces of constant Gauss curvature. However, different analytic behaviors of $\k$ will induce different global behaviors on the class of $\k$-surfaces. As a matter of fact, we construct some examples that do not exist in when the Gauss curvature is constant; for example, complete $\k$-surfaces of strictly negative Gauss curvature.

We next detail the organization of the paper and highlight some of the main results.

In Section \ref{sec:phaseplane}, we introduce the phase plane of the solutions of the ODE fulfilled by the profile curve of a rotational surface satisfying \eqref{defKsuprot}. First, in Section \ref{subsec:diffeq}, we derive those differential equations and deduce the non-linear autonomous system defined by them. The space of solutions of this differential system allows us to properly define the phase plane in Section \ref{subsec:phaseplane}, where we exhibit its main properties. In Section \ref{subsec:radialsols}, we address the problem of proving the existence of a radial graph intersecting orthogonally the axis of rotation that solves Eq. \eqref{defKsuprot}. Since the differential equations are singular at the axis of rotation, standard theory cannot be invoked in order to ensure the existence such a graph. Instead, we apply a fixed point argument on a differential operator to overcome these difficulties.

In Section \ref{sec:rotacionales}, we classify rotational $\k$-surfaces by analyzing their geometric properties through the qualitative study of the solutions of the phase plane. In the same fashion as for surfaces of constant Gauss curvature, we distinguish cases on $\k$: in Section \ref{subsec:kanula}, we assume that $\k$ vanishes at some point; in Section \ref{subsec:kneg}, we assume that $\k$ is negative; and in Section \ref{subsec:kpos}, we assume that $\k$ is positive. Under mild hypotheses on $\k$, the examples obtained are a generalization of their constant Gauss curvature counterparts. We also prove the existence of $\k$-surfaces that do not exist when $\k$ is constant, such as complete graphs that are either entire or asymptotic to a vertical cylinder, or $\k$-surfaces of strictly negative Gauss curvature.

Finally, in Section \ref{sec:asympt}, we study the asymptotic behavior of the complete graphs defined in Section \ref{subsec:kanula}. First, we deduce that polynomials $x^n$ rotated around the $z$-axis generate $\k$-surfaces for adequate choices of $\k$, in terms of $n$. Then, we prove that depending on how $\k$ vanishes at the origin, a complete graph intersecting orthogonally the axis of rotation can either converge to a vertical cylinder or be entire.

\section{A phase plane study}\label{sec:phaseplane}
Throughout this paper we consider $\K\in C^1(\S^2)$ and $\k\in C^1([-1,1])$ related by \eqref{dependenciarotacional}. In the present section we deduce the differential equations fulfilled by a rotational $\k$-surface around the $z$-axis and introduce the phase plane, the main tool to derive the qualitative properties of their solutions.
\subsection{The differential equations.}\label{subsec:diffeq}

Given an arc-length parametrized curve $\alpha(s)=(x(s),0,z(s))$, $s\in I\subset\R$, we parametrize a rotational surface $\sig$ around the $z$-axis, $\{(0,0,t);\ t\in\R\}$, by
$$
\psi(s,\theta)=(x(s)\cos\theta,x(s)\sin\theta,z(s)): I\times(0,2\pi)\rightarrow\R^3.
$$
Up to a change of the orientation, the angle function is given by $\nu(\psi(s,\theta))=x'(s)$, and the principal curvatures of $\sig$ are
\begin{equation}\label{princcurv}
\kappa_1=\kappa_\alpha(s)=x'(s)z''(s)-x''(s)z'(s),\hs5 \kappa_2=\frac{z'(s)}{x(s)},
\end{equation}
where $\kappa_\alpha(s)$ is the curvature of $\alpha(s)$ as a planar curve. The arc-length condition $x'(s)^2+z'(s)^2=1$ ensures the existence of a $2\pi$-periodic function $\theta:I\rightarrow\R$ such that $x'(s)=\cos\theta(s),\ z'(s)=\sin\theta(s)$. Geometrically, $\theta(s)$ is the angle between $\alpha'(s)$ and $e_1$. Therefore, from Eq. \eqref{princcurv} we deduce that the following system holds
\begin{equation}\label{sistemadiferencial}
\left\lbrace\begin{array}{l}
\vspace{.25cm} x'(s)=\cos\theta(s)\\
\vspace{.25cm} z'(s)=\sin\theta(s)\\
\theta'(s)=\displaystyle{\frac{x(s)\k(\cos\theta(s))}{\sin\theta(s)}},
\end{array}\right.
\end{equation}
on every $J\subset I$ where $\theta(s)\neq n\pi,\ n\in\Z$.


Note that if $(x(s),z(s),\theta(s))$ solves system \eqref{sistemadiferencial}, then $(x(s),-z(s),2\pi-\theta(s))$ is also a solution. So, we can restrict the image of $\theta(s)$ to lie in $(0,\pi)$ and as a consequence $z'(s)>0$, i.e. the height function of $\alpha(s)$ is always increasing. Hence, bearing in mind that $z'$ is defined by means of $x$ and $\theta$, \eqref{sistemadiferencial} can be simplified as the autonomous ODE system given by 
\begin{equation}\label{1ordersys}
	\left(\begin{array}{c}
		x(s)\\
		\theta(s)
	\end{array}\right)'=\left(\begin{array}{c}
		\cos\theta(s)\\
		\displaystyle{\frac{x(s)\k(\cos\theta(s))}{\sin\theta(s)}}
	\end{array}\right),\quad \textrm{with}\; \theta(s)\in(0,\pi).
\end{equation}

\subsection{The phase plane.}\label{subsec:phaseplane}
The phase plane of \eqref{sistemadiferencial} is $\Theta:=[0,\infty)\times(0,\pi)$ with coordinates denoting, respectively, the distance to the axis of rotation and the angle of $\alpha'(s)$ with the $e_1$-direction. The orbits $\gamma(s)=(x(s),\theta(s))$ are the solutions of \eqref{sistemadiferencial}. 


As a first approach to the understanding of the phase plane, we compile some of its main properties.

\begin{lem}\label{lemapropiedades}
The following properties are immediate from the definition of the phase plane.
\begin{enumerate}
\item If $\k(\cos\t_0)=0,\ \t_0\neq\pi/2$, the orbit $(s,\t_0),\ s\geq0$, corresponds to a flat cone of constant angle function $\cos\t_0$. 
\item If $\k(0)=0$, the constant orbit $(x_0,\pi/2),\ x_0>0$ corresponds to the flat cylinder of radius $x_0$ and vertical rulings. In this case, the set $E_0=\{(x_0,\pi/2),\ x_0>0\}$ will be referred as the \emph{equilibria of system \eqref{1ordersys}}.
\item If $\k(1)=0$ (resp. $\k(-1)=0$), horizontal planes with upwards (resp. downwards) orientation are $\k$ surfaces. The curve $(s,0)$ (resp. $(s,\pi)$) for $s\geq0$ in the boundary of $\Theta$ will be regarded as the orbit corresponding to such horizontal plane. Moreover, since $\k$ has $C^1$-regularity, no orbit in $\Theta$ can intersect the boundary component $(s,0)$ (resp. $(s,\pi)$), $s>0$.
\item For each $(x_0,\t_0)\in\Theta$, the Cauchy problem of \eqref{1ordersys} for this initial data has existence and uniqueness. Therefore, the orbits provide a foliation by regular $C^1$ curves of $\Theta$ (or $\Theta-E_0$, if $E_0\neq\varnothing$). As a matter of fact, if an orbit $\gamma(s)$ converges to some $e_0\in E_0$ then $s\rightarrow\pm\infty$.
\end{enumerate}
\end{lem}

Next we focus on the motion of an orbit in $\Theta$. Observe that $\gamma'(s_0)$ is vertical if and only if $\theta(s_0)=\pi/2$. In particular, $\gamma$ can be written as a vertical graph $\theta=\theta(x)$ outside the line $\theta=\pi/2$. The lines $\theta=\theta_0$ such that $\theta_0=\pi/2$ or $\k(\cos\theta_0)=0$ divide $\Theta$ into connected components where the coordinates of an orbit are monotonous functions. Hence, at each of these \emph{monotonicity regions}, the motion of $\gamma$ is uniquely determined as detailed next:

\begin{pro}\label{comportamientoorbita}
In the above conditions, for any $(x_0,\theta_0)\in\Theta-E_0,\ \k(\cos\t_0)\neq0$, the following properties hold:
\begin{itemize}
\item[1.] If $\theta_0 \in(0,\pi/2)$ and $\k(\cos\theta_0)>0$ (resp. $\k(\cos\theta_0)<0$), then $\theta(x)$ is strictly increasing (resp. decreasing) at $x_0$.
\item[2.] If $\theta_0 \in(\pi/2,\pi)$ and $\k(\cos\theta_0)>0$ (resp. $\k(\cos\theta_0)<0$), then $\theta(x)$ is strictly decreasing (resp. increasing) at $x_0$.
\item[3.] If $\theta_0=\pi/2$, then $\gamma$ is orthogonal to the line $\theta=\theta_0$.
\end{itemize}
\end{pro}

%

The following result discusses whether an orbit $\gamma(s)$ can \emph{diverge} in $\Theta$.

\begin{pro}\label{orbitaescapa}
Let be $\k\in C^1([-1,1])$, $\gamma(s)=(x(s),\t(s))$ an orbit in $\Theta$ and suppose that $x(s)\rightarrow\infty$ and $\t(s)\rightarrow\t_0\in[0,\pi]$ as $s\rightarrow\pm\infty$. Then, $\k(\cos\t_0)=0$.
\end{pro}

\begin{proof}
We suppose $s\rightarrow\infty$, since the case $s\rightarrow-\infty$ is analogous. From the monotonicity properties of $\Theta$ we deduce the existence of some $s_0$ such that for $s\geq s_0$, $\gamma(s)$ is contained in some monotonicity region of $\Theta$. At this monotonicity region we can express $\theta$ as a function of $x$, $\theta(x)$. The chain rule yields
\begin{equation}\label{eqthetagrafo}
\t'(x)\cos\t(x)=\frac{d\t}{dx}\frac{dx}{ds}=\frac{d\t}{ds}=\t'(s)=\frac{x\k(\cos\t(x))}{\sin\t(x)}.
\end{equation}
On the one hand, by hypothesis, when $x\rightarrow\infty$ we have $\theta(x)\rightarrow\t_0$ and by the mean value theorem and by monotonicity $\t'(x)\rightarrow0$. On the other hand, after making $x\rightarrow\infty$ in Eq. \eqref{eqthetagrafo} we conclude
$$
0=\frac{''\infty''\k(\cos\t_0)}{\sin\t_0}.
$$
So, $\k(\cos\t_0)=0$ in order for this limit to be well-defined.
\end{proof}

\subsection{Existence of radial solutions.}\label{subsec:radialsols}
We finish this section by studying $\k$-surfaces intersecting the axis of rotation. Let be $\t_0\in[0,\pi]$. If $\t_0\neq0,\pi$, then the existence and uniqueness of the Cauchy problem for system \eqref{sistemadiferencial} yields the existence of an orbit having the point $(0,\t_0)$ as endpoint. This orbit generates a profile curve that intersects the axis of rotation at a cusp point with angle function equal to $\cos\t_0$. In particular, the surface is not complete since it fails to be $C^1$ at this intersection point.

The two missing points of the phase plane are $(0,0)$ and $(0,\pi)$. An orbit having any of such points as endpoint corresponds to a $\k$-surface intersecting orthogonally the axis of rotation. The existence of such orbits follows easily from the existence of a rotational graph satisfying \eqref{defKsuprot}. We prove the existence of such a graph next.

Let us express a rotational $\k$-surface as the radial graph
\begin{equation}\label{paramgraforot}
(x\cos\theta,x\sin\theta,u(x)),\ x>0.
\end{equation}
With this parametrization, the angle function is $1/\sqrt{1+u'(x)^2}$ and Eq. \eqref{defKsuprot} writes as
\begin{equation}\label{kcomografo}
\frac{u''u'}{x(1+u'^2)^2}=\k\left(\frac{1}{\sqrt{1+u'^2}}\right).
\end{equation}
Multiplying by $x$ and after integration, this equation transforms into
$$
\displaystyle{\left(\frac{u'^2}{1+u'^2}\right)}'=2x\k\left(\frac{1}{\sqrt{1+u'^2}}\right)
$$
The existence of a rotational $\k$-surface intersecting orthogonally the axis of rotation is equivalent to establish the existence of a classical solution of
\begin{equation}\label{eqsingular}
\left\lbrace
\begin{array}{ll}
\displaystyle{\left(\frac{u'^2}{1+u'^2}\right)}'=2x\k\left(\frac{1}{\sqrt{1+u'^2}}\right)&\mathrm{in}\ (0,\delta),\\
u(0)=0,\hspace{.25cm} u'(0)=0,&
\end{array}
\right.
\end{equation}
for some $\delta>0$. First, observe that Eq.  \eqref{eqsingular} is singular at $x=0$, hence we cannot apply standard theory to ensure its existence. Second, the condition $u(0)=0$ is not restrictive at all because our problem is invariant under vertical translations. Finally, the condition $u'(0)=0$ is just the fact that the graph intersects orthogonally the axis of rotation.

\begin{teo}\label{existenciaradial}
Let be $\k\in C^1([-1,1])$ such that $\k(1)\geq0$. The initial value problem \eqref{eqsingular} has a unique solution $u\in C^2([0,\delta]),\ u\geq0$, for some $\delta>0$.
\end{teo}

\begin{proof}
The arguments for the proof of this result are inspired by the study of radial solutions to certain divergence-type equations; see \cite{CCO} and also Prop. 5 in \cite{Lop}. 

We define
$$
f(y)=\frac{y}{1+y},\hspace{.5cm} \phi(y)=2\k\left(\frac{1}{\sqrt{1+y}}\right).
$$
Fix $\delta>0$ to be determined later. A function $u(x)\in C^2([0,\delta])$ is a solution of \eqref{eqsingular} if and only if
$$
(f(u'(x)^2))'=x\phi(u'(x)^2),\hspace{.5cm} u(0)=0,\ u'(0)=0.
$$
If we write $u'(x)^2=g(x)$, then the above equation is
\begin{equation}\label{eqg}
(f(g(x))'=x\phi(g(x)),\hspace{.5cm} g(0)=0.
\end{equation}

After solving $g(x)$ in \eqref{eqg}, we define the operator $\mathfrak{T}$ by
$$
(\mathfrak{T}g)(x)=f^{-1}\left(\int_0^x t\phi(g(t))dt\right).
$$
Note that the existence of a fixed point of $\mathfrak{T}$ is equivalent to the existence of a solution of \eqref{eqg}.

At this point, since $f^{-1}$ and $\phi$ are Lipschitz continuous in $[-\varepsilon,\varepsilon]$ for $0<\varepsilon<1$, arguing as in the proof of Proposition 5 in \cite{Lop} we conclude that $\mathfrak{T}$ is a contraction in the space $(C^0([0,\delta]),||.||_\infty)$, after choosing $\delta>0$ small as necessary. Hence, we ensure the existence of $g\in C^0([0,\delta])$ such that $\mathfrak{T}g=g$.  Now, since Eq.  \eqref{eqg} has a unique $C^1$ solution for the initial data $g(x_0)=g_0,\ x_0>0$, the function $g$ lies in the space $C^0([0,\delta])\cap C^1((0,\delta])$.

Since $u'(x)^2=g(x)$ we conclude that $u'(x)=\pm\sqrt{g(x)}$. We assume $u'(x)=\sqrt{g(x)}$, i.e. $u(x)\geq0$, since the case $u'(x)=-\sqrt{g(x)}$, which leads to $u(x)\leq0$, is similar after a reflection and a change of the orientation. 

At this point, $u(x)$ is a solution of \eqref{eqsingular} provided that it belongs to $C^2([0,\delta])$, i.e. that it has $C^2$-regularity at $x=0$. Taking limit in Eq. \eqref{kcomografo} as $x\rightarrow0$ and applying the L'Hôpital rule yields
$$
u''(0)^2=\k(1).
$$
Since $\k(1)\geq0$, the value $u''(0)$ is well-defined and thus $u\in C^2([0,\delta])$. This proves Th. \ref{existenciaradial}.
\end{proof}

\begin{obs}
For the existence of a radial solution intersecting orthogonally the axis of rotation with unit normal $-e_3$, we argue as before. The only difference is to change the parametrization given by \eqref{paramgraforot} to 
$$
((\delta-x)\cos\t,(\delta-x)\sin\t,u(x)),\hspace{.5cm} x\in[0,\delta],\ u'(\delta)=0,
$$
whose induced normal at the axis of rotation is $-e_3$. Details are skipped.
\end{obs}

The existence of radial solutions of Eq. \eqref{defKsuprot} provided by Th. \ref{existenciaradial} has the following consequence on the phase plane.
 
\begin{cor}\label{existenciaorbitamasmenos}
Let be $\k\in C^1([-1,1])$. If $\k(1)\geq0$ (resp $\k(-1)\geq0$), there exists an orbit $\gamma_+$ (resp. $\gamma_-$) having the point $(0,0)$ (resp. $(0,\pi)$) as endpoint.
\end{cor}

\section{Rotational $\k$-surfaces}\label{sec:rotacionales}
Once the properties of the phase plane have been introduced, we aim to give a classification of rotational $\k$-surfaces. We divide this study by analyzing three possible behaviors of $\k$: it vanishes at some point, it is positive, or it is negative. These cases are inspired in the classification of rotational surfaces of constant Gauss curvature.

\subsection{Case where $\k$ vanishes.}\label{subsec:kanula} In this section, we assume that $\k\in C^1([-1,1])$ vanishes.

Recall that when $\k$ is constantly zero, the rotational flat surfaces are cones, cylinders and horizontal planes, all of them parametrized by their angle function. Indeed, cones have constant angle varying between $0$ and $\pm1$, limit cases corresponding to cylinders and horizontal planes, respectively. If $\nu\neq0$ we denote by $C_{\nu_0}$ to the cone of constant angle $\nu_0$.

As remarked in Lem. \ref{lemapropiedades} and Prop. \ref{comportamientoorbita}, both the local and global behavior of $\k$ fully determine the structure of the phase plane. Nonetheless, although a general classification result for rotational $\k$-surfaces for an arbitrary $\k$ seems hopeless, a common example arises when assuming that $\k$ vanishes at some point, as we prove next.

\begin{pro}\label{existenciagrafoscompletosconos}
Let be $\k\in C^1([-1,1])$ such that $\k(\nu_0)=0$ for some $\nu_0\in(-1,1)$. If $\nu_0\geq0$ and $\k(1)>0$ (resp. $\nu_0\leq0$ and $\k(-1)>0$), there exists a rotational, strictly convex, upwards (resp. downwards) oriented, entire $\k$-graph $\sig_{\nu_0}$. Moreover, if $\nu_0\neq0$, after a vertical translation $\sig_{\nu_0}$ is asymptotic to $C_{\nu_0}$.
\end{pro}

\begin{proof}
We prove the case that $\nu_0\geq0$ and $\k(1)>0$, since the case $\nu_0\leq0,\ \k(-1)>0$, is similar. Let $\widehat{\nu_0}$ be the largest value in $[0,1]$ where $\k$ vanishes and define $\t_0=\arccos\widehat{\nu_0}$. For saving notation, we will keep naming $\widehat{\nu_0}$ as $\nu_0$.

By the definition of $\nu_0$ and since $\k(1)>0$, $\k_{|[\nu_0,1]}$ is non-negative and only vanishes at $\nu_0$. As a matter of fact, $\Theta_+=\{(x,\t)\in\Theta;\ \t<\t_0\}$ is a monotonicity region and due to Prop. \ref{comportamientoorbita}, an orbit $\gamma(s)=(x(s),\t(s))$ in $\Theta_+$ satisfies $x'(s)>0$ and $\t'(s)>0$.

Let $\gamma_+(s)$ be the orbit in $\Theta$ with $\gamma_+(0)=(0,0)$ as endpoint, given by Cor. \ref{existenciaorbitamasmenos}. This orbit generates a planar, arc-length parametrized curve $\alpha_+(s)$ in $\r3$ intersecting the $z$-axis orthogonally with upwards unit normal. Thus, $\gamma_+(s)$ lies in $\Theta_+$ for $s>0$ small enough; see Fig. \ref{fig:clasicono}, left.

\begin{figure}[h]
\centering
\includegraphics[width=.85\textwidth]{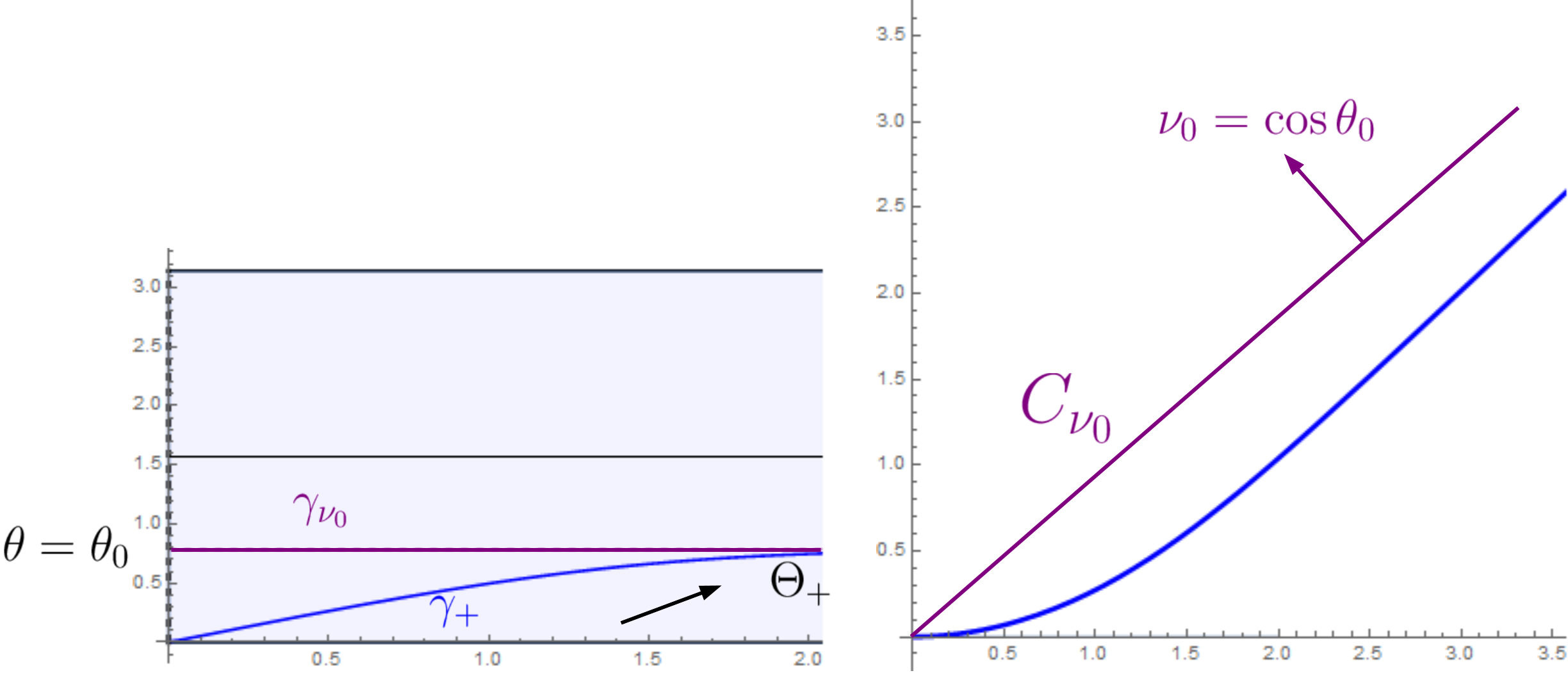}
\caption{Left: The phase plane for a function $\k$ such that $\k(\t_0)=0,\ \t_0<\pi/2$ and $\k(1)>0$, and the orbit $\gamma_+$. Right, the profile curve of the entire graph associated to $\gamma_+$.}
\label{fig:clasicono}
\end{figure}

At this point, we distinguish cases on $\nu_0$. First, assume that $\nu_0\neq0$ (i.e. $\t_0\neq\pi/2$) and define $\gamma_{\nu_0}(s)=(s,\t_0),\ s\geq0,$ the orbit corresponding to the cone $C_{\nu_0}$. We prove that $\gamma_+$ cannot converge to some $(x_0,\t_0)\in \gamma_{\nu_0}$. Arguing by contradiction, suppose that $\gamma_+(s)\rightarrow(x_0,\t_0)\in \gamma_{\nu_0}$, which necessarily holds as $s\rightarrow\infty$ by the uniqueness of the Cauchy problem. Since $x(s)\rightarrow x_0$ as $s\rightarrow\infty$, by the mean value theorem we conclude that $x'(s)\rightarrow0$ as $s\rightarrow\infty$. However, this is a contradiction since $x'(s)=\cos\t(s)$ and $\t(s)\rightarrow\t_0\neq\pi/2$. Therefore, $\gamma_+$ can be expressed a graph $\t(x)$ in $\Theta_+$ satisfying $\t(0)=0,\ \t'(x)>0$ and $\t(x)<\t_0$ for every $x>0$. In particular, by Prop. \ref{orbitaescapa} we have $\lim_{x\rightarrow\infty}\t(x)=\t_0$, that is $\t(x)$ has $\gamma_{\nu_0}$ as horizontal asymptote and converges to it.

If $\nu_0=0$, the orbit $\gamma_+$ satisfies $\t(s)\rightarrow\pi/2$ as $s\rightarrow\infty$ and $x(s)$ can either behave as: $x(s)\rightarrow x_0>0$ or $x(s)\rightarrow\infty$. In Section \ref{sec:asympt} we study in detail these two possibilities.

In any case, the curve $\alpha_+(s)$ in $\r3$ defined by $\gamma_+(s)$ is complete and intersects the $z$-axis orthogonally. Moreover, if $\nu_0\neq0$, up to a vertical translation, $\alpha_+$ converges to $C_{\nu_0}$ (see Fig. \ref{fig:clasicono}, right); otherwise, it either converges to a flat cylinder of a certain radius or is an entire graph. The surface obtained by rotating $\alpha_+$ around the $z$-axis is a $\k$-surface with the properties announced in the statement of Prop. \ref{existenciagrafoscompletosconos}.

This completes the proof for the case that $\nu_0\geq0$. The case $\nu_0<0$ is treated similarly.
\end{proof}

In the case that $\k(1)=0$ (resp. $\k(-1)=0$), horizontal planes with upwards (resp. downwards) orientation are $\k$-surfaces. Depending on the properties of $\k$, one can find different behaviors in the class of $\k$-surfaces. For instance, the next cases are easy to prove using the properties of the phase plane, although details will be skipped here.
\begin{ex}\label{example1}
Let be $\k\in C^1([-1,1])$ such that $\k(\pm1)=0$.
\begin{itemize}
\item If $\k(y)>0,\ \forall y\in(-1,1)$, then for every $x_0>0$ there exists an orbit passing through $(x_0,\pi/2)$ and having two points at $x=0$ as endpoints. Consequently, the $\k$-surfaces are a 1-parameter family of compact, non-complete surfaces intersecting the axis of rotation at singular cusp points.
\item If $\k(y)<0,\ \forall y\in(-1,1)$, then for every $x_0>0$ there exists an orbit passing through $(x_0,\pi/2)$ and whose ends are asymptotic to the lines $x=0$ and $x=\pi$ . Consequently, the $\k$-surfaces associated to these orbits are a 1-parameter family of properly embedded annuli that resemble to the usual minimal catenoids of negative Gauss curvature.
\end{itemize}
\end{ex}

\subsection{The case $\k<0$.}\label{subsec:kneg} This section is devoted to classify rotational $\k$-surfaces for the case that $\k<0$. Note that since $\k$ is $C^1$, there exists $c_0<0$ such that $\k\leq c_0<0$, and in particular none of the following $\k$-surfaces is complete, in virtue of Efimov's theorem: no $C^2$-surface can be immersed in $\R^3$ if its Gauss curvature is bounded from above by a negative constant \cite{Efi}.

The classification result proved next is a generalization of the well-known case of constant Gauss curvature $K<0$.

\begin{teo}\label{clasificacionKnegativa}
Let be $\k\in C^1([-1,1])$ such that $\k<0$. The rotational $\k$-surfaces are classified as follows:
\begin{itemize}
\item[1.] For each $x_0>0$, a $\k$-surface diffeomorphic to $\S^1\times(0,1)$ having as boundary circumferences of certain radii, and whose \emph{waist} (the smallest circumference) has radius $x_0$.
\item[2.] For each $\nu_0\in(-1,1),\ \nu_0\neq0$, a simply connected $\k$-surface that intersects the axis of rotation with angle function equal to $\nu_0$, and has as boundary a circumference of a certain radius.
\item[3.] Two $\k$-surfaces diffeomorphic to $\S^1\times(0,1)$. One end of is unbounded and converges to the $z$-axis, and the other has as boundary a circumference of a certain radius.
\end{itemize}
Moreover, the boundary circumferences consist on singular points that do not belong to any $\k$-surface.
\end{teo}

\begin{proof}
Let be $\k\in C^1([-1,1])$ such that $\k<0$. The structure of the phase plane is the following: there are two monotonicity regions, namely $\Theta_1=\Theta\cap\{\t<\pi/2\}$ and $\Theta_2=\Theta\cap\{\t>\pi/2\}$, with monotonicity directions given by Prop. \ref{comportamientoorbita}.

First, we prove Item $\mathit{1}$. Let be $x_0>0$ and $\gamma_{x_0}(s)$ the orbit in $\Theta$ such that $\gamma_{x_0}(0)=(x_0,\pi/2)$. For $s>0$ (resp. $s<0$), $\gamma_{x_0}(s)$ lies in $\Theta_1$ (resp. in $\Theta_2$); we focus in the case $s>0$ since the opposite is similar.

By monotonicity, $\gamma_{x_0}(s)$ is written as a vertical graph $g_{x_0}(x),\ x\geq x_0$ such that $g_{x_0}(x_0)=\pi/2$ and $g_{x_0}'(x)<0$. Note that $g_{x_0}$ cannot satisfy $g_{x_0}(x)\rightarrow\t_0$ as $x\rightarrow\infty$ in virtue of Prop. \ref{orbitaescapa}, and cannot converge to a finite point $(x_0,\t_0)\in\Theta_1$. So, the only possibility is that $\gamma_{x_0}(s)\rightarrow(x_0^1,0)$ as $s\rightarrow s_0$ for some $x_0^1>x_0$. For $s<0$ the discussion is similar, i.e. $\gamma_{x_0}(s)$ converges to some $(x_0^2,\pi)$ with $x_0^2>x_0$.

Hence, for each $x_0>0$, $\gamma_{x_0}(s)$ is a bi-graph over $\{(x,\pi/2),\ x\geq 0\}$ that has $(x_0^1,0)$ and $(x_0^2,\pi)$ as endpoints. The corresponding $\k$-surface is diffeomorphic to $\S^1\times(0,1)$, is a bi-graph over some horizontal plane and has as boundary two circles of singular points of radii $x_0^1$ and $x_0^2$ ; see Fig. \ref{fig:clasineg}, the orbit and profile curve in orange.

\begin{figure}[h]
\centering
\includegraphics[width=.7\textwidth]{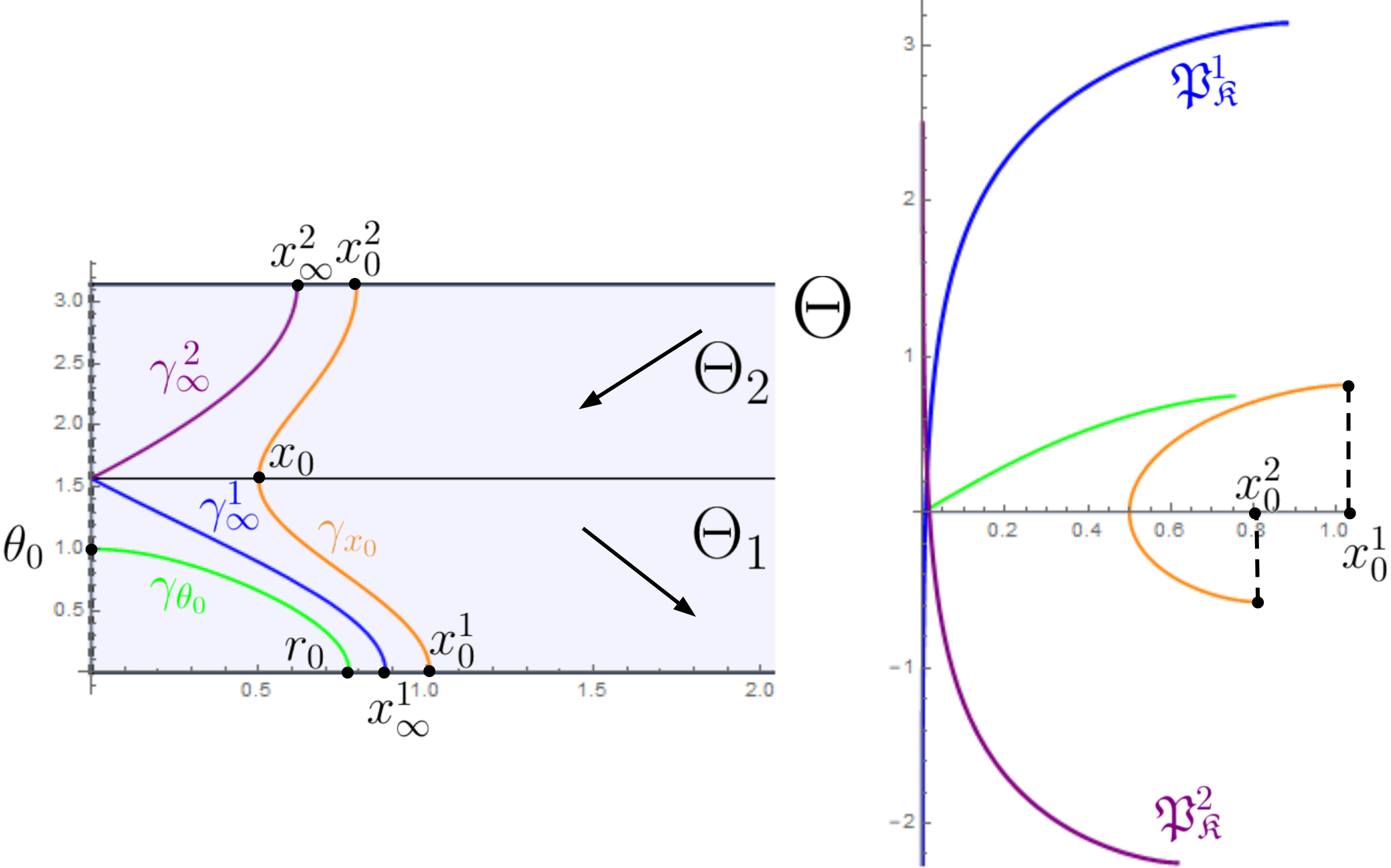}
\caption{Left: the phase plane and the different orbits for a function $\k<0$. Right: the profile curve associated to each orbit.}
\label{fig:clasineg}
\end{figure}

For the proof of Item $\mathit{2}$, let be $\nu_0\in(0,1)$, define $\t_0=\arccos\nu_0$ and consider $\gamma_{\t_0}(s)$ the orbit in $\Theta$ having at $s=0$ the point $(0,\t_0)$ as endpoint; the case $\nu_0\in(-1,0)$ is similar. By monotonicity, $\gamma_{\t_0}(s)$ lies in $\Theta_1$ for $s>0$ and stays there while it converges to some finite point $(r_0,0)$. The corresponding $\k$-surface is a simply connected graph having as boundary a circle of singular points of radius $r_0$; see Fig. \ref{fig:clasineg}, the orbit in green.

Lastly, we exhibit the existence of two $\k$-surfaces with the properties announced in Item $\mathit{3}$. The following claims will be needed in order to achieve this proof:

\begin{claim}
If $\gamma(s)=(x(s),\t(s))$ is an orbit in $\Theta$ converging to the point $(0,\pi/2)$, the parameter $s$ must tend to $\pm\infty$
\end{claim}
\begin{proofclaim}
Recall that the orbit $\gamma$ is a solution of
\begin{equation}\label{sistode}
\left(\begin{array}{c}
x(s)\\
\theta(s)
\end{array}\right)'
=
\left(\begin{array}{c}
\cos\theta(s)\\
\displaystyle{\frac{x(s)\k(\cos\theta(s))}{\sin\theta(s)}}
\end{array}\right):=F(x,\t).
\end{equation}
The function $F(x,\t)$ is $C^1$ at $(0,\pi/2)$, hence there exists a unique solution of \eqref{sistode} with this initial data. Indeed, the explicit solution is given by $x(s)=0,\ \t(s)=\pi/2$. We emphasize that this solution does not generate a $\k$-surface since the corresponding arc-length parametrized curve is $x(s)=0,\ z(s)=s$, which agrees with the axis of rotation. By uniqueness of the Cauchy problem of \eqref{sistode} for the initial condition $(0,\pi/2)$, an orbit in $\Theta$ cannot have as finite endpoint such initial condition, hence if $\gamma(s)$ converges to $(0,\pi/2)$, it does with $s\rightarrow\pm\infty$. This proves \textbf{Claim 1}.
\end{proofclaim}

Consider a strictly decreasing sequence $x_n>0,\ x_n\rightarrow0$, and let $\gamma_{x_n}(s)=(x_n(s),\t_n(s))$ be the orbit in $\Theta$ such that $\gamma_{x_n}(0)=(x_n,0)$. As defined in the proof of Item $\textit{1}$ for $s>0$, we define $x_n^1>x_n$ the value for which $(x_n^1,0)$ is the endpoint of $\gamma_{x_n}$.

\begin{claim}
\emph{The sequence $x_n^1$ as defined above satisfies $x_n^1\rightarrow x_\infty>0$}. 
\end{claim}
\begin{proofclaim}
By their definition, and since two distinct orbits in $\Theta$ cannot intersect, it is clear that if $x_n<x_m$ then $x_n^1\leq x_m^1$. So, $x_n^1$ is a decreasing sequence of positive numbers. We claim that this sequence is bounded from below above zero. Indeed, take some $\t_0\in(0,\pi/2)$ and consider the orbit $\gamma_{\t_0}(s)$ having $(0,\t_0)$ as endpoint. The discussion made in the proof of Item $\textit{2}$ ensures us that $\gamma_{\t_0}$ converges to some $(r_0,0),\ r_0>0$. Since $\gamma_{\t_0}$ and each $\gamma_{x_n}$ are disjoint orbits, it is clear that $r_0\leq x_n^1,\ \forall n$. In conclusion, $x_n^1\rightarrow x_\infty^1\geq r_0$ as $n\rightarrow\infty$ and in particular $x_\infty^1>0$. This proves \textbf{Claim 2}.
\end{proofclaim}

So, as $n\rightarrow\infty$ the orbits $\gamma_{x_n}$ converge to some orbit $\gamma_\infty^1$ having the points $(0,\pi/2)$ and $(x_\infty^1,0)$ as limit endpoints. Moreover, $\gamma_\infty^1$ converges to $(0,\pi/2)$ as $s\rightarrow -\infty$; see Fig. \ref{fig:clasineg}, the orbit in blue.

Finally, let $\alpha_\infty^1(s)$ be the arc-length parametrized curve associated to $\gamma_\infty^1(s)$ and denote by $\mathfrak{P}_\k^1$ to the $\k$-surface generated by $\alpha_\infty^1$. When $s\rightarrow-\infty$, $x(s)\rightarrow0$ and $z'(s)\rightarrow1$, and the length of $\alpha_\infty^1(s)$ tends to infinity, that is $\alpha_\infty^1(s)$ converges to the $z$-axis. So, $\mathfrak{P}_\k^1$ has the topology of $\S^1\times(0,1)$, with one end converging to the $z$-axis and the other having as boundary a circle of singular points of radius $x_\infty^1$.

The $\k$-surface $\mathfrak{P}_\k^2$ is defined in the same fashion, but taking the orbits $\gamma_{x_n}(s)$ such that $\gamma_{x_n}(0)=(x_n,0)$ and making $s<0$. Hence, each $\gamma_{x_n}(s)$ lies in $\Theta_2$ and this time the limit orbit $\gamma_\infty^2$ converges to $(0,\pi/2)$ as $s\rightarrow\infty$.
\end{proof}
The $\k$-surfaces $\mathfrak{P}_\k^i,\ i=1,2,$ are the analogous to the pseudosphere of constant curvature $K=-1$.

\begin{obs}
Assume that $\k(y)<0,\ \forall y\in(-1,1)$ and $\k(\pm1)=0$. Then, as concluded in Ex. \ref{example1}, the orbits of type $\gamma_{x_0}$ and $\gamma_\infty^i,\ i=1,2,$ are asymptotic to the lines $x=0$ and $x=\pi$. In particular, all the corresponding $\k$-surfaces are complete, and none of them are a contradiction with Efimov's theorem since their Gauss curvature tend to zero. The orbits of type $\gamma_{\t_0}$ are also asymptotic to the lines $x=0$ and $x=\pi$, generating entire graphs having a cusp point at the axis of rotation.
\end{obs}


\subsection{The case $\k>0$.}\label{subsec:kpos} In this section we focus on the case $\k>0$ and aim to generalize the case of constant curvature $K>0$. Among the surfaces of positive constant curvature, the most highlighted without any doubt is the round sphere. It is the only complete surface of constant positive curvature, has different characterizations, and has been widely used to prove several results in many geometric frameworks.

Regarding $\k$-surfaces, if $\K(\eta)=\k(\langle\eta,e_3\rangle)$ is rotationally symmetric, a straightforward change of variable in Eq.  \eqref{condicionMinkowski} yields that a necessary and sufficient condition on $\k$ for the existence of a $\k$-sphere is
\begin{equation}\label{condicionMinkowski1dim}
\int_{-1}^1\frac{y}{\k(y)}=0.
\end{equation}
The next result generalizes the classification of constant positive curvature surfaces, provided that $\k>0$ satisfies this integral condition.

\begin{teo}\label{clasificacionkpositiva}
Let be $\k\in C^1([-1,1],\ \k>0,$ satisfying \eqref{condicionMinkowski1dim}. Then, rotational $\k$-surfaces are classified as follows:
\begin{enumerate}
\item A strictly convex sphere intersecting orthogonally the axis of rotation. This sphere is unique among compact, immersed $\k$-surfaces of genus 0 in $\R^3$.
\item A 1-parameter family of compact, non-complete surfaces intersecting the axis of rotation at two cusp points. Moreover, the parameter defining this family is the angle of intersection with the axis of rotation.
\item A 1-parameter family of non-complete annuli, whose ends converge to two circles of singular points. Moreover, the parameter defining this family is the maximum distance to the axis of rotation.
\end{enumerate}
\end{teo}

\begin{proof}
Since $\k>0$, the structure of the phase plane is as follows: there exist two monotonicity regions $\Theta_1=\{\t<\pi/2\}$ and $\Theta_2=\{\t>\pi/2\}$ where an orbit $\gamma=(x,\t)$ satisfies $x'>0$ and $x'<0$, respectively. Since $\k>0$ it follows that $\t'>0$ at both regions.

The existence of a strictly convex sphere follows immediately from the Minkowski theorem \cite{Min}. The fact that the sphere is rotational follows by applying Alexandrov reflection technique with respect to vertical planes. The uniqueness of this sphere follows from the outstanding work of \cite{GaMi2}, see Th. 1.6. Therefore, there exists an orbit $\gamma_+$ having $(0,0)$ and $(0,\pi)$ as endpoints. As a matter of fact, this orbit intersects the line $\t=\pi/2$ at some $(x_+,\pi/2)$ with $x_+>0$. See Fig. \ref{fig:clasipos}, the orbit in blue.

Next, fix some $\t_0\in(0,\pi/2)$. Then, there exists an orbit $\gamma_{\t_0}$ having $(0,\t_0)$ as endpoint. Since $\gamma_{\t_0}$ and $\gamma_+$ cannot intersect each other, $\gamma_{\t_0}$ ends up at the line $x=0$. The corresponding $\k$-surface is compact and intersects the axis of rotation at two cusp points. See Fig. \ref{fig:clasipos}, the orbit in orange.

Finally, fix some $x_0>x_+$. The orbit $\gamma_{x_0}$ passing through $(x_0,\pi/2)$ must converge to the lines $\t=0$ and $\t=\pi$, without intersecting them, and with decreasing $x$-coordinate. Therefore, the corresponding $\k$-surface is homeomorphic to an annulus whose ends converge to circles of singular points. Moreover, the maximum distance of this annulus to the axis of rotation is precisely $x_0$. See Fig. \ref{fig:clasipos}, the orbit in purple.

This completes the proof of Th. \ref{clasificacionkpositiva}.
\end{proof}

\begin{figure}[h]
\centering
\includegraphics[width=.45\textwidth]{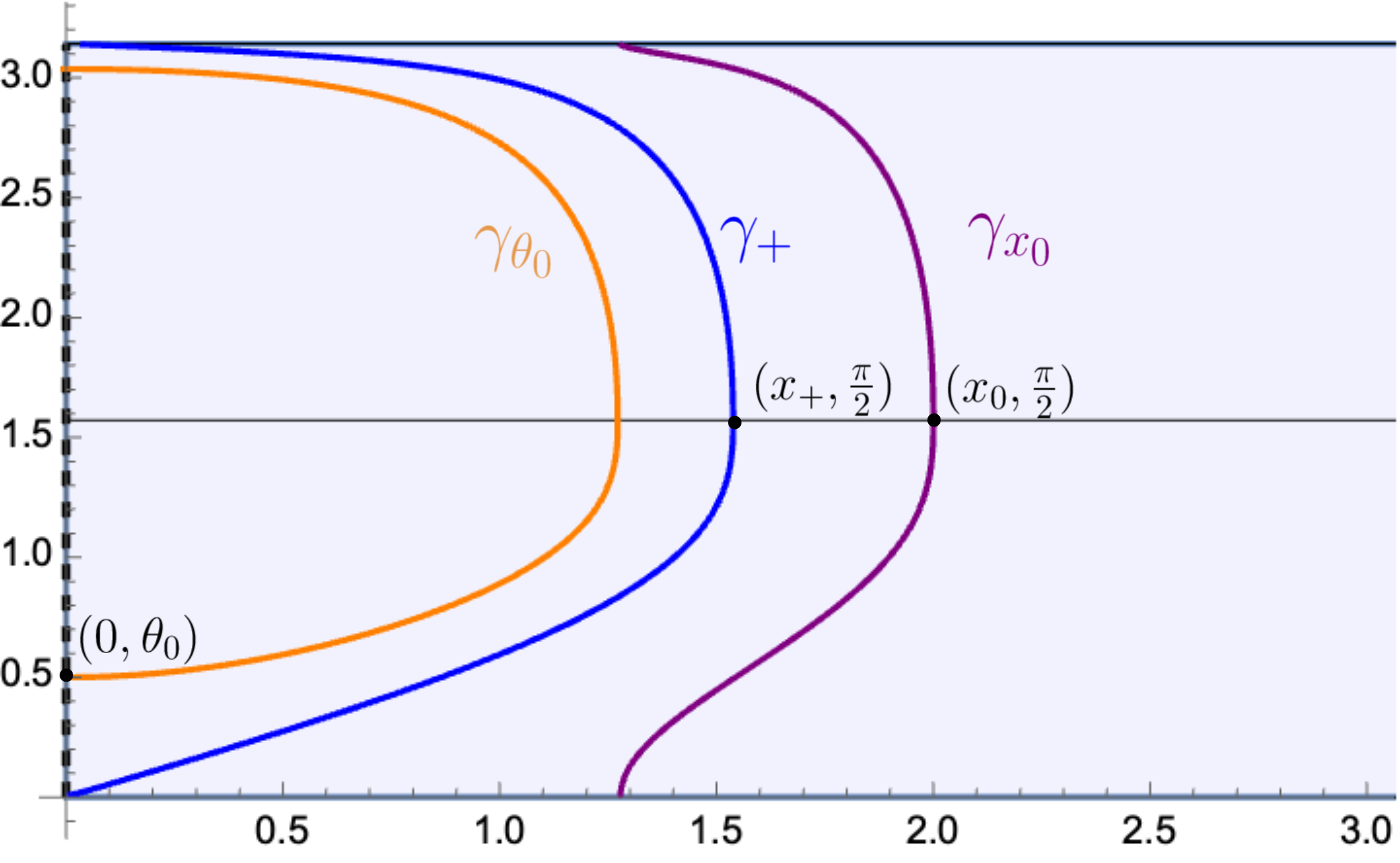}\hspace{1cm}\includegraphics[width=.21\textwidth]{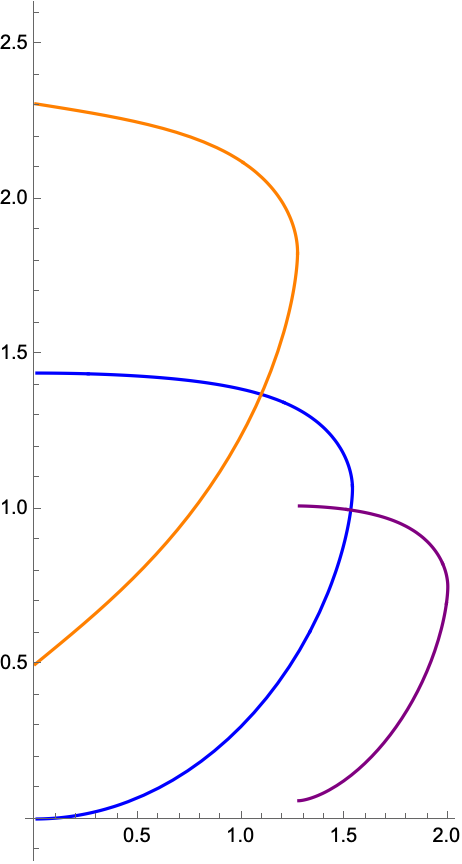}
\caption{Left: the phase plane and the different orbits for a function $\k>0$ such that Eq. \eqref{condicionMinkowski1dim} holds. Note the lack of symmetry of the phase plane with respect to $\t=\pi/2$, since $\k$ fails to be even. Right: the profile curve associated to each orbit.}
\label{fig:clasipos}
\end{figure}

For the case that the integral condition \eqref{condicionMinkowski} fails, the different possible behaviors of $\k$ determine the properties of the corresponding $\k$-surfaces. For example, it can be proved (although details are skipped at this point) that the $\k$-surfaces for the function $\k(y)=y+2$ are as described next:
\begin{enumerate}
\item The orbit starting at $(0,0)$ ends up at some $(0,\t_0)$ with $\t_0>\pi/2$. Therefore, the corresponding $\k$-surface has a singular cusp point at the axis of rotation.
\item The orbit ending at $(0,\pi)$ starts at some $(x_0,0)$, with $x_0>0$, without intersecting it. Therefore, the corresponding $\k$-surface converges to a circle of singular points.
\item Any other orbit intersecting the line $x=0$ generates a $\k$-surface intersecting the axis of rotation at a cusp point.
\item Any other orbit not being one of the aforementioned ones generates a $\k$-surface homeomorphic to an annulus whose ends converge to circles of singular points.
\end{enumerate}
In particular, none of these $\k$-surfaces are complete, in contrast with the case of positive, constant curvature.

\section{The asymptotic behavior of complete rotational $\k$-graphs}\label{sec:asympt}
In Section \ref{subsec:kanula} we proved that if $\k(y)>0,\ \forall y\in(0,1)$ and $\k(0)=0$ then the rotational $\k$-surface intersecting orthogonally the axis of rotation is a strictly convex graph that either converges to a flat cylinder or is entire. In this section, we study in detail this behavior.

Let $u:[0,R_0)\rightarrow\R,\ R_0\leq\infty$ be a $C^2$ positive function, and let $\sig_u$ be the surface obtaining by rotating the graph $(x,0,u(x))$ around the $z$-axis. Our first goal is to prove that surfaces obtained by rotating polynomials $x^n,\ n\geq2$, are $\k$-surfaces for adequate choices of the prescribed function.

\begin{pro}\label{comportamientopolinomios}
Let be $u_n(x)=x^n,\ n\geq2$, and consider $\sig_{u_n}$ the upwards oriented surface in $\r3$ obtained by rotating the graph of $u_n(x)$ around the $z$-axis. Then, $\sig_{u_n}$ is a $\k_n$-surface for the prescribed function
$$
\k_n(y)=(n-1)y^{\frac{2n}{n-1}}(1-y^2)^{\frac{n-2}{n-1}}.
$$
\end{pro}

\begin{proof}
Let be $u_n(x),\ x\geq0$, and $\sig_{u_n}$ as in the hypothesis of Prop. \ref{comportamientopolinomios}. According to Eq.  \eqref{kcomografo}, the Gauss curvature of $\sig_{u_n}$ is well-defined at $x=0$ if and only if $n\geq2$, and its value is zero. 

Since $u_n(x)$ is strictly convex for $x>0$, $\sig_{u_n}$ is a $\k$-surface for some 1-dimensional function $\k_n(y)$ such that $\k_n(y)>0,\ \forall y\in(0,1)$. We derive next the explicit expression of $\k_n$. The angle function $\nu_n(x)$ of $\sig_{u_n}$ is
\begin{equation}\label{nupoliniomio}
\nu_n(x)=\frac{1}{\sqrt{1+n^2x^{2(n-1)}}},
\end{equation}
and so Eq.  \eqref{kcomografo} reads as
$$
K_{\sig_{u_n}}=n^2(n-1)x^{2(n-2)}\nu_n(x)^4.
$$
Solving $n^2x^{2(n-1)}$ from Eq.  \eqref{nupoliniomio} in terms of $\nu_n$ and substituting yields
$$
K_{\sig_{u_n}}=\frac{(n-1)\nu_n(x)^2(1-\nu_n(x)^2)}{x^2}.
$$
Finally, by solving $x^2$ from Eq.  \eqref{nupoliniomio} in terms of $\nu_n$ and substituting we conclude
$$
K_{\sig_{u_n}}=(n-1)\nu_n(x)^{\frac{2n}{n-1}}(1-\nu_n(x)^2)^{\frac{n-2}{n-1}},
$$
that is, $K_{\sig_{u_n}}$ is expressed in terms of $\nu_n$. Hence, by defining
\begin{equation}\label{kene}
\k_n(y)=(n-1)y^{\frac{2n}{n-1}}(1-y^2)^{\frac{n-2}{n-1}},
\end{equation}
we conclude that $\sig_{u_n}$ is a $\k_n$-surface for the prescribed function $\k_n$.
\end{proof}

\begin{obs}
Recall that these surfaces do not depend $C^1$ on the angle function at the intersection with the axis of rotation. Indeed, at $y=1$ the function $\k_n$ behaves as $(1-y)^{\frac{n-2}{n-1}}$, which fails to be $C^1$. Nevertheless, for the behavior at infinity we are only interested at the behavior of $\k_n$ at $y=0$, where it has $C^1$-regularity.
\end{obs}

Note that $\k_n(0)=0$ and so $K_{\sig_{u_n}}\rightarrow0$ as $x\rightarrow\infty$. Furthermore, $\k_n(y)$ behaves at $y=0$ as 
$$
\lim_{y\rightarrow0}\frac{\k_n(y)}{y^{\boldsymbol{\alpha}(n)}}=n-1>0,\hs5 \boldsymbol{\alpha}(n):=\frac{2n}{n-1}.
$$

The following result gives an explicit expression for rotational graphical $\k$-surfaces when $\k(y)=y^\alpha$, and its proof follows easily by explicit integration of Eq. \eqref{kcomografo}.
\begin{pro}\label{comportamientoyelevadoalpha}
Let be $\k(y)=y^\alpha,\ \alpha>0$, and $u_\alpha(x),\ x\geq0$, the function defining the graph of the rotational $\k$-surface intersecting orthogonally the axis of rotation with upwards orientation.
\begin{itemize}
\item[1.] If $\alpha>2$, then
$$
u_\alpha(x)=\int_0^x\sqrt{\left(\frac{\alpha-2}{2}t^2+1\right)^{\frac{2}{\alpha-2}}-1}\ dt.
$$
In particular, $u_\alpha(x)$ is defined for every $x\geq0$, $\sig_{u_\alpha}$ is an entire graph, and
$$
\lim_{x\rightarrow\infty}\frac{u_\alpha(x)}{x^{\boldsymbol{n}(\alpha)}}=C>0,\hs5 \boldsymbol{n}(\alpha):=\frac{\alpha}{\alpha-2}.
$$
\item[2.] If $\alpha=2$, then
$$
u_2(x)=\int_0^x\sqrt{e^{\frac{t^2}{2}}-1}\ dt.
$$
In particular, $u_2(x)$ is defined for every $x\geq0$ and $\sig_{u_2}$ is an entire graph.
\item[3.] If $\alpha<2$, then
$$
u_\alpha(x)=\int_0^x\sqrt{\left(\frac{2R_\alpha}{2R_\alpha-t^2}\right)^{\frac{2-\alpha}{2}}-1}\ dt,
$$
where $R_\alpha:=1/(2-\alpha)$. In particular, $u_\alpha(x)$ is defined for $x\in[0,\sqrt{2R_\alpha})$ and $\sig_{u_\alpha}$ converges to the flat cylinder of radius $\sqrt{2R_\alpha}$ and vertical rulings.
\end{itemize}
\end{pro}

We emphasize that Prop. \ref{comportamientopolinomios} and \ref{comportamientoyelevadoalpha} are in some sense back and forth. Specifically:
\begin{itemize}
\item Given $n>2$, the rotation of the graph $u_n(x)=x^n$ is a $\k$-surface for the function $\k_n$ given by \eqref{kene}, which has the same behavior at $y=0$ as the function $\k(y)=y^{\boldsymbol{\alpha}(n)}$, where $\boldsymbol{\alpha}(n)=2n/(n-1)$.
\item Given $\alpha>2$ and the function $\k(y)=y^\alpha$, the rotational $\k$-surface intersecting orthogonally the axis of rotation is given by the graph of the function $u_\alpha(x)$ defined in Item $\mathit{1}$. of Prop. \ref{comportamientoyelevadoalpha}, which behaves for $x\rightarrow\infty$ as the polynomial $x^{\boldsymbol{n}(\alpha)}$, where $\boldsymbol{n}(\alpha)=\alpha/(\alpha-2)$
\item The functions $\boldsymbol{n}$ and $\boldsymbol{\alpha}$ are one inverse of the other, i.e. $\boldsymbol{n}(\boldsymbol{\alpha}(n))=n$ and $\boldsymbol{\alpha}(\boldsymbol{n}(\alpha))=\alpha$.
\end{itemize}

\def\refname{References}

\vspace{1cm}

The first author was partially supported by P18-FR-4049. For the second author, this research is a result of the activity developed within the framework of the Programme in Support of Excellence Groups of the Región de Murcia, Spain, by Fundación Séneca, Science and Technology Agency of the Región de Murcia. Irene Ortiz was partially supported by MICINN/FEDER project PGC2018-097046-B-I00 and Fundación Séneca project 19901/GERM/15, Spain.
\end{document}